\documentclass[12pt]{amsart}
\usepackage{amsfonts,amscd,latexsym,amsmath,amssymb,cancel,enumerate,mathrsfs}
\theoremstyle{plain}
\newtheorem{master}{Master}[section]

\newtheorem{thm}[master]{Theorem}
\newtheorem{fact}[master]{Fact}

\newtheorem{question}[master]{Question}

\theoremstyle{definition}
\newtheorem{defin}[master]{Definition}

\theoremstyle{remark}
\newtheorem{remark}[master]{Remark}
\numberwithin{equation}{section}

\newcommand{\Rea}{\mathbb{R}}
\newcommand{\Nat}{\mathbb{N}}
\newcommand{\Rat}{\mathbb{Q}}

\newcommand{\Norm}{\|\cdot \|}

\begin{document}
\title[Non-commutative uniform Banach group]{An example of a non-commutative uniform Banach group}
\author[M. Doucha]{Michal Doucha}
\address{Institute of Mathematics\\ Polish Academy of Sciences\\
00-656 Warszawa, Poland}
\email{m.doucha@post.cz}
\thanks{The author was supported by IMPAN's international fellowship programme partially sponsored by PCOFUND-GA-2012-600415.}
\begin{abstract}
Benyamini and Lindenstrauss mention in their monograph \emph{Geometric nonlinear functional analysis Vol. 1.,
American Mathematical Society Colloquium Publications, 48. American Mathematical Society, Providence, RI, 2000} that there is no known example of a non-commutative uniform Banach group. Prassidis and Weston also asked whether there is a non-commutative example. We answer this problem affirmatively. We construct a non-commutative uniform Banach group which has the free group of countably many generators as a dense subgroup.

Moreover, we show that our example is a free one-generated uniform Banach group whose metric induced by the norm is bi-invariant.
\end{abstract}
\maketitle
\section*{Introduction}
A uniform Banach group is a Banach space equipped with an additional group structure so that the group unit coincides with the Banach space zero and the group operations are uniformly continuous with respect to norm. Uniform Banach groups were introduced and studied by Enflo in \cite{Enf1}, \cite{Enf2} with connection to the infinite-dimensional version of the Hilbert's fifth problem. Typical example comes when we are given two Banach spaces $X$ and $Y$ and a uniform homeomorphism $\phi:X\rightarrow Y$ between them such that $\phi(0)=0$. Then we can define a (commutative) group operation $\cdot$ on $X$ as follows: for $x,y\in X$ we set $x\cdot y=\phi^{-1}(\phi(x)+\phi(y))$. Note that unless $\phi$ is linear there is no a priori connection between the two group operations $+$ (resp. $+_X$) and $\cdot$.

A comprehensive source of information about uniform Banach groups is Chapter 17 in \cite{BeLi}. As mentioned there, the following problem was left open. \emph{Does there exist a non-commutative uniform Banach group?} This question was also asked by Prassidis and Weston in \cite{PraWes} and \cite{PraWes2}. Here we give a positive answer to this question. The following is the main result.
\begin{thm}\label{main}
There exists an infinite dimensional separable Banach space $(\mathbb{X},+,0,\Norm)$ equipped with an additional group structure $(\cdot,{}^{-1},0)$ whose unit coincides with the Banach space zero, the group multiplication $\cdot$ is invariant with respect to the norm $\Norm$, and $F_\infty$, the free group of countably many generators, is a dense subgroup of $(\mathbb{X},\cdot,{}^{-1},0)$.

In particular, there exists a non-commutative uniform Banach group.
\end{thm}
\section{Preliminaries}
 We assume the reader to know basic facts about uniform spaces. We refer to Chapter 8 in \cite{Eng} for more information.

Recall that for any topological group $G$ there are two distinguished compatible uniformities: the left uniformity $\mathcal{U}_L$ generated by basic entourages of the form $\{(g,h):g^{-1}h\in U\}$, where $U$ is a basic neighborhood of the identity in $G$; and the right uniformity $\mathcal{U}_R$ which is generated by basic entourages of the form $\{(g,h):hg^{-1}\in U\}$, where $U$ is again a basic neighborhood of the identity in $G$.

We start with the following definition given by Enflo in \cite{Enf1}.
\begin{defin}
Let $G$ be a topological group. $G$ is called \emph{uniform} if there exists a compatible uniformity on $G$ such that the group multiplication is uniformly continuous with respect to that uniformity.
\end{defin}

Below, we collect some basic facts about uniform groups.\\
\begin{fact}
~\\
\begin{enumerate}
\item (see Proposition 1.1.3. in \cite{Enf1}) A topological group $G$ is uniform if and only if the left and right uniformities coincide.
\item (folklore) That is in turn equivalent with the fact that there exists a neighborhood basis of the unit of $G$ consisting of open sets closed under conjugation. Such groups are more often called \emph{SIN} (small invariant neighborhood) groups, or also \emph{balanced} groups.
\item (folklore) In case that $G$ is metrizable, i.e. the neighborhood basis can be taken countable, G is uniform, resp. SIN, if and only if it admits a compatible bi-invariant metric; i.e. metric $d$ such that for any $x,y,a,b\in G$ it holds that $d(x,y)=d(axb,ayb)$ (the same reasoning gives that if $G$ is not metrizable then its topology is given by a family of bi-invariant pseudometrics).\\

\end{enumerate}
\end{fact}

\noindent {\bf Examples:}
\begin{enumerate}[(a)]
\item All abelian and compact topological groups are uniform groups. It is obvious for the former. For the latter, consider a compact topological group $G$. Notice that by the continuity of the group operations, for any open neighborhood $U$ of the identity and for any group element $g$ there are open neighborhoods $V_g$ of $g$ and $W_g\subseteq U$ of the identity so that $V_g\cdot W_g\cdot V_g^{-1}\subseteq U$. Then by compactness one can find finitely many elements $g_1,\ldots,g_n\in G$ so that $V_{g_1},\ldots,V_{g_n}$ cover $G$. Take $W=\bigcap_{i\leq n} W_{g_i}$ and notice that for any $g\in G$ we have $g\cdot W\cdot g^{-1}\subseteq U$. Clearly, $\bigcup_{g\in G} g\cdot W\cdot g^{-1}$ is then a conjugacy-invariant open neighborhood of the identity contained in $U$.
\item The Heisenberg group $UT_3^3(\mathbb{R})$ consisting of the upper triangular $3\times 3$-matrices is not uniform. Note that the Heisenberg group is very close to both being abelian and compact. It is a locally compact group of nilpotency class 2. Since $UT_3^3(\mathbb{R})$ is metrizable it is sufficient to show that it does not admit a compatible bi-invariant metric. Suppose for contradicition that $d$ is a compatible bi-invariant metric on $UT^3_3(\mathbb{R})$. Given matrices

\  $A'=\begin{pmatrix}
1 & a' & c'\\
0 & 1 & b'\\
0 & 0 & 1\\
\end{pmatrix}$, $A=\begin{pmatrix}
1 & a & c\\
0 & 1 & b\\
0 & 0 & 1\\
\end{pmatrix}$,   
a computation gives that 

$$A^{-1}A'A=\begin{pmatrix}
1 & a' & c'-ab'+a'b\\
0 & 1 & b'\\
0 & 0 & 1\\
\end{pmatrix}.$$ 

Since $d$ is compatible there exist $r>0$ and $R>0$ such that the open ball of radius $R$ with respect to $d$ centred at $I$, the identity matrix, is contained in the open set of matrices $\begin{pmatrix}
1 & x & z\\
0 & 1 & y\\
0 & 0 & 1\\
\end{pmatrix}$ with $|z| < r$. Since $d$ is compatible, we can choose $A'$ with $a' \neq 0$ and    $d(A',I)<R$. Then by choosing  $b$ sufficiently large we can guarantee that $|c'-ab'+a'b|>r$. It follows that $d(A^{-1}A'A,I)\geq R$. So $d(A',I)\neq d(A^{-1}A'A,I)$, and thus $d$ is not bi-invariant.
\item However, an example of a uniform group that is finite-dimensional Euclidean is the group of unitary $n\times n$-matrices. That follows either from the fact that this group is compact, or one can check that the Hilbert-Schmidt norm there induces a compatible bi-invariant metric.
\item Let $G$ be a group that acts faithfully by isometries on a bounded metric space $(M,d_M)$. The following is then a bi-invariant metric (that makes $G$ a topological uniform group): for $g,h\in G$ set $$d(g,h)=\sup_{x\in M} d_M(gx,hx).$$

Every group with bi-invariant metric is of this form. Indeed, if $G$ is a group with bi-invariant metric $d_G$, then $G$ has a faithfull action by isometries on itself induced by left translations. The formula above defining a metric gives the original metric $d_G$.

\end{enumerate}

\begin{defin}[Enflo, \cite{Enf2}]
A uniform group $G$ is \emph{Banach} if it is uniformly homeomorphic to some Banach space.
\end{defin}
The preceding definition is readily checked to be equivalent with that one mentioned in the introduction; i.e. a Banach space with additional group structure, where the additional group operations are uniformly continuous with respect to norm, and the additional group unit coincides with the Banach space $0$.

Recall the canonical example of a uniform Banach group given by some uniform homeomorphism between Banach spaces that was also mentioned in the introduction. Let us mention here the result of Enflo which says that under some conditions the converse is true. We refer to \cite{Enf1} and \cite{Enf2} for unexplained notions from the statement of the theorem.
\begin{thm}[Enflo, \cite{Enf2}]
Let $G$ be a commutative uniform Banach group such that the corresponding Banach space has roundness $p>1$ and such that moreover:
\begin{itemize}
\item $G$ is {\it uniformly dissipative},
\item for every $x_1,x_2,y\in G$ we have $\|x_1\cdot y-x_2\cdot y\|=o(\|x_1-x_2\|^{1/p})$ uniformly in $x_1,x_2,y$ as $\|x_1-x_2\|\to 0$.

\end{itemize}

Then $G$ is isomorphic to the additive group of some Banach space.
\end{thm}
\subsection{Preliminary discussion of the proof}
Before proceding to the proof of Theorem \ref{main}, let us roughly explain the ideas behind it. We shall construct a countable set $X$ equipped with two group structures. Under the first group structure, $X$ is isomorphic to the free group of countably many generators. Under the second one, $X$ is isomorphic to the minimal subgroup of the real vector space with countable Hamel basis that contains the elements of this Hamel basis and is closed under multiplication by scalars that are dyadic rationals. Moreover, $X$ will get equipped with a metric which is bi-invariant with respect to the non-commutative group operation and which behaves like a norm with respect to the latter group operation. In particular, the completion of $X$ with respect to this metric will become a real Banach space.

Let us mention two `peculiarities' of the construction. First, $X$ is constructed so that there is no connection between these two algebraic structures (i.e. non-commutative group structure and commutative group structure with dyadic scalar multiplication) in the sense that for any $x,y\in X$, the elements $x,y,x\cdot y,x^{-1}$ are linearly independent (in the commutative structure), and similarly, for any $x,y\in X$ and dyadic rationals $\alpha,\beta$ we have that $\alpha x+\beta y$ is a new free group generator (in the non-commutative structure).

The second thing to mention is that the metric on $X$ is constructed by induction on \emph{finite} fragments $X_n$, $n\in \Nat$, of $X$. This will give us a better control of the metric we are defining.

The construction is inspired and in some sense analogous to the construction of Lipschitz-free Banach spaces (we refer the reader to the book \cite{Wea} as the main reference on this subject) and also to the construction of free groups with the Graev metric (see \cite{Gr}). Indeed, the uniform Banach group we construct can be viewed as a free object in an appropriate category and we refer the reader to the last section of the paper where this is discussed. We also discuss there in more detail the similarity between our Banach group and Lipschitz-free Banach spaces and groups with the Graev metric.
\subsection{Some notation regarding free groups and vector spaces}
Let $S$ be some non-empty set and let $w$ be a word over the alphabet $S\cup S^{-1}\cup\{e\}$, where $S^{-1}=\{s^{-1}:s\in S\}$ is a disjoint copy of $S$ interpreted as a set of inverses of elements from $S$ and $e$ is an element not belonging to either $S$ or $S^{-1}$ which shall be interpreted as a group unit. We say that $w=w_1\ldots w_n$ is \emph{irreducible} if either $n=1$ and $w=w_1=e$, or $n>1$ and for every $i\leq n$, $w_i\neq e$ and there is no $i<n$ such that $w_i=w_{i+1}^{-1}$. If $w$ is a word that is \emph{not} irreducible, then by $w'$ we shall denote the unique irreducible word obtained from $w$ by deleting each occurence of the letter $e$ and each occurence of neighbouring letters $a$ and $a^{-1}$ (if this procedure leads to an empty word, then we set $w'$ to be $e$).

It is well-known and easy to observe that elements of $F(S)$, the free group of free generators coming from the set $S$ with $e$ as a unit, are in one-to-one correspondence with irreducible words over the alphabet $S\cup S^{-1}\cup\{e\}$.

Let $n\in \Nat$. By $W_n(S)$ we shall denote the set of all irreducible words of length at most $n$ over the alphabet $S\cup S^{-1}\cup\{e\}$.\\

Let now similarly $B$ be some non-empty set not containing the distinguished element $e$. The vector space over some field $F$ with $B$ as the maximal linearly independent set and $e$ representing zero can be viewed as a set of all functions from $B$ to $K$ that have finite support, where $f$ has finite support if for all but finitely many $b\in B$ we have $f(b)=0$.

In our case, we shall work with $F=\Rea$ (resp. $\Rat$), however since we shall need to work with only finitely many vectors at any given time we restrict to functions whose range is some specified finite subset of $\Rea$ (also, the set $B$ will be always finite).

Let $K\subseteq_\mathrm{fin} \Rea$ some finite subset of reals. Then by $V_K(B)$ we shall denote the set of all functions from $B$ to $K$ with finite support. The requirement on finite support will be in our construction superfluous since $B$ will be at any given time finite as already mentioned.

\section{The proof of Theorem \ref{main}}
We start by describing the underlying countable set $X$ mentioned above in the preliminaries.
\subsection{The underlying dense set}
We now describe a countably infinite set $X$, constructed as an increasing union of finite sets $X_0\subseteq X_1\subseteq \ldots$, which will also carry a multiplicative group operation and the corresponding group inverse operation so that $X$ will be isomorphic to a free group of countably many generators, and it will also carry an additive (abelian) group operation together with multiplication by scalars that are dyadic rationals so that it is a proper subgroup of an infinite dimensional vector space over the rationals. Moreover, the unit for addition and multiplication will be the same.

Later, we shall define a metric on $X$ that will be invariant under both addition and multiplication and will preserve scalar multiplication. The completion then will be a Banach space over the reals which is also equipped with multiplication with free group as a dense part.\\

Let $n\geq 1$ be arbitrary. By $D_n$ we shall denote the set of dyadic rationals $\frac{a}{2^n}$, where $a\in [-2^{2n},2^{2n}]$. Clearly, $D=\bigcup_n D_n$ is the set of all dyadic rational numbers.

We set $X_0=\{e\}$. Let $S_1=\{x\}$ be some singleton. We set $X_1=W_1(S_1)=\{e,x,x^{-1}\}$. Let $B_2=X_1\setminus X_0=\{x,x^{-1}\}$. We set $X_2=V_{D_1}(B_2)=\{\alpha x+\beta x^{-1}:\alpha,\beta\in D_1\}$.\\

Suppose now that $X_{2n}$ has been constructed. We need to construct $X_{2n+1}$ and $X_{2n+2}$. Set $S_{2n+1}$ to be $S_{2n-1}\cup (X_{2n}\setminus X_{2n-1})$. Then we set $X_{2n+1}$ to be $W_{2n+1}(S_{2n+1})$.

Next we set $B_{2n+2}$ to be $B_{2n}\cup (X_{2n+1}\setminus X_{2n})$. Then we set $X_{2n+2}$ to be $V_{D_{n+1}}(B_{2n+2})$.

This finishes the inductive construction. Note that $X=\bigcup_n X_n=\bigcup_n W_{2n+1}(S_{2n+1})=\bigcup_n V_{D_{n+1}}(B_{2n+2})$. It follows that if we set $S=\bigcup_n S_{2n+1}$ and $B=\bigcup_n B_{2n+2}$, then $X$ is also naturally isomorphic to $F(S)$, the free group of countably many generators coming from the set $S$, and it is a (additive) proper subgroup of the rational (or real) vector space with $B$ as the maximal linearly independent set - the minimal subgroup that contains free abelian group with $B$ as a set of generators that is closed under multiplication by scalars from $D$.\\

We define inductively a rank function $r:X\to \omega$. For any $x\in X$ we set $r(x)=0$ if neither $x$ nor $x^{-1}$ is possible to write as $\alpha_1 y_1+\ldots+\alpha_m y_m$, where $m>1$, $\alpha_i>0$, for every $i$, and $\alpha_1+\ldots+\alpha_m=1$.

If $x$ or $x^{-1}$ \emph{is} possible to write as $\alpha_1 y_1+\ldots+\alpha_m y_m$, where $m>1$, $\alpha_i>0$, for every $i$, and $\alpha_1+\ldots+\alpha_m=1$, then we set $r(x)=\max\{r(y_i):i\leq m\}+1$.\\

\subsection{Construction of the metric}
We shall now define a metric $\rho$ and a norm $\Norm$ on $X$. Actually, the metric and the norm will be one and the same in the sense that for any $x,y\in X$ we shall have $\rho(x,y)=\|x-y\|$ and $\|x\|=\rho(x,e)$. The distinguishing is done only for practical notational reasons since we understand $X$ as both a free group and a subgroup of a vector space. In the former case, it is more natural to consider a metric there, while in the latter to consider a norm there.
By induction, we shall define functions $\rho_n:X_n^2\rightarrow \Rea$ (the range will actually be a subset of non-negative rationals) for odd $n$ and functions $\Norm_n:X_n\rightarrow \Rea$ for even $n$ that satisfy the following properties:
\begin{enumerate}
\item\label{con1} for every odd $n$ we have that $\rho_n$ is a symmetric function that is equal to zero only on diagonal, i.e. $\rho_n(x,y)=0$ iff $x=y$ for $x,y\in X_n$; similarly, for every even $n$ we have that $\|x\|_n=0$ iff $x=e$ for $x\in X_n$,
\item\label{con2} for every even $n$, $\Norm_n$ extends $\rho_{n-1}$, i.e. for every $a,b\in X_{n-1}$ we have $$\|a-b\|_n=\rho_{n-1}(a,b),$$ similarly for every odd $n$, $\rho_n$ extends $\Norm_{n-1}$, i.e. for every $a,b\in X_{n-1}$ such that $a-b\in X_{n-1}$ we have $$\rho_n(a,b)=\|a-b\|_{n-1},$$
\item\label{con3} for every odd $n$, for any words (not necessarily irreducible) $w_1,w_2,v_1,v_2$ over the alphabet $S_n\cup S_n^{-1}\cup\{e\}$ such that\\ $w'_1,w'_2,v'_1,v'_2,(w_1w_2)',(v_1v_2)'\in X_n$, we have $$\rho_n((w_1w_2)',(v_1v_2)')\leq \rho_n(w'_1,v'_1)+\rho_n(w'_2,v'_2),$$ and for every $a,b\in X_n$ we have $$\rho_n(a,b)=\rho_n(a^{-1},b^{-1}),$$
\item\label{con4} for every odd $n$ and for every $a\in X_n$ and for every $b\in X_n$ such that $b=\alpha_1 c_1+\ldots+\alpha_m c_m$, where $m>1$, $\alpha_i\geq 0$ and $c_i\in B_{n-1}$, for every $i$, and $\alpha_1+\ldots+\alpha_m=1$, we have $$\rho_n(a,b)\leq \alpha_1\cdot\rho_n(a,c_1)+\ldots+\alpha_m\cdot \rho_n(a,c_m),$$
\item\label{con5} for every even $n$ and for every $a\in X_n$ and any scalar $\alpha$ such that $\alpha a\in X_n$ we have $$\|\alpha a\|_n=|\alpha|\cdot \|a\|_n,$$ and for every $a,b\in X_n$ such that also $a+b\in X_n$ we have $$\|a+b\|_n\leq \|a\|_n+\|b\|_n,$$
\item\label{con6} for every odd $n$, for every $a\in X_n$ and for every $b\in X_n$ such that $b=(\alpha_1 c_1+\ldots+\alpha_m c_m)^{-1}$, where, $m>1$, $\alpha_i\geq 0$ and $c_i\in B_{n-1}$, for every $i$, and $\alpha_1+\ldots+\alpha_m=1$, we have $$\rho_n(a,b)\leq \alpha_1\cdot\rho_n(a,c_1^{-1})+\ldots+\alpha_m\cdot \rho_n(a,c_m^{-1}),$$ and similarly, for every even $n$, for every $a\in X_n$ and for every $b\in X_{n-1}$ such that $b=(\alpha_1 c_1+\ldots+\alpha_m c_m)^{-1}$, where $m>1$, $\alpha_i\geq 0$ and $c_i\in B_{n-2}$, for every $i$, and $\alpha_1+\ldots+\alpha_m=1$, we have $$\|a-b\|_n\leq \alpha_1\cdot\|a-c_1^{-1}\|_n+\ldots+\alpha_m\cdot \|a-c_m^{-1}\|_n.$$\\

\end{enumerate}
We define $\rho_1$ on $X_1$ by setting $\rho_1(x,e)=\rho_1(x^{-1},e)=1$ and $\rho_1(x,x^{-1})=2$. Obviously, this satisfies all the requirements.

Let us now define $\Norm_2$ on $X_2$, i.e. we have to define $\|\alpha x+\beta x^{-1}\|_2$ for every $\alpha,\beta\in D_1$. We set $$\|\alpha x+\beta x^{-1}\|_2=\min\{|\gamma_1|\cdot \rho_1(x,e)+|\gamma_2|\cdot \rho_1(x^{-1},e)+|\gamma_3|\cdot \rho_1(x,x^{-1}):$$ $$\alpha x+\beta x^{-1}=\gamma_1 x+\gamma_2 x^{-1}+\gamma_3(x-x^{-1}),\gamma_1,\gamma_2,\gamma_3\in\Rea\}.$$

Condition \eqref{con6} is automatically satisfied as there is no $b\in X_1$ of the form $(\alpha_1 c_1+\ldots+\alpha_m c_m)^{-1}$ for appropriate $\alpha$'s and $c$'s. Condition \eqref{con1} is also obvious. Thus we need to check the conditions \eqref{con2} and \eqref{con5}. 

Let us do the former. We need to check that $\Norm_2$ extends $\rho_1$. We shall check that $\|x\|_2=\rho_1(x,e)$. The other cases are similar. Clearly, $\|x\|_2\leq \rho_1(x,e)$. Suppose that $\|x\|_2=|\gamma_1|\cdot \rho_1(x,e)+|\gamma_2|\cdot \rho_1(x^{-1},e)+|\gamma_3|\cdot \rho_1(x,x^{-1})$, where $x=\gamma_1 x+\gamma_2 x^{-1}+\gamma_3(x-x^{-1})$. It follows that necessarily $\gamma_2=\gamma_3$ and $\gamma_1+\gamma_3=1$. By triangle inequality we have $|\gamma_2|\cdot\rho_1(x^{-1},e)+|\gamma_2=\gamma_3|\cdot\rho_1(x,x^{-1})\geq |\gamma_2|\cdot \rho_1(x,e)$, thus $|\gamma_1|\cdot \rho_1(x,e)+|\gamma_2|\cdot \rho_1(x^{-1},e)+|\gamma_3|\cdot \rho_1(x,x^{-1})\geq |\gamma_1|\cdot \rho_1(x,e)+|\gamma_2|\cdot \rho(x,e)\geq |\gamma_1+\gamma_2|\cdot \rho_1(x,e)=\rho_1(x,e)$.

We now check \eqref{con5}. Fix some $a\in X_2$ and $\alpha$ such that $\alpha a\in X_2$. Suppose that $\|a\|=|\gamma_1|\cdot \rho_1(x,e)+|\gamma_2|\cdot \rho_1(x^{-1},e)+|\gamma_3|\cdot \rho_1(x,x^{-1})$, where we have $a=\gamma_1 x+\gamma_2 x^{-1}+\gamma_3(x-x^{-1})$. Then since $\alpha a=\alpha\cdot \gamma_1 x+\alpha\cdot\gamma_2 x^{-1}+\alpha\cdot\gamma_3(x-x^{-1})$ we have that $\|\alpha a\|_2\leq |\alpha\gamma_1|\cdot \rho_1(x,e)+|\alpha\gamma_2|\cdot \rho_1(x^{-1},e)+|\alpha\gamma_3|\cdot \rho_1(x,x^{-1})=|\alpha|\cdot \|a\|_2$. The other inequality is analogous. By a similar argument one can also show that for any $a,b\in X_2$ such that also $a+b\in X_2$ we have $\|a+b\|_2\leq \|a\|_2+\|b\|_2$ and we leave this to the reader. In fact, we note here that this definition of $\Norm_2$ is equivalent with that one that says that $\Norm_2$ is the greatest function that satisfies condition \eqref{con4} and such that $\|x\|_2\leq \rho_1(x,e)$, $\|x^{-1}\|_2\leq \rho_1(x^{-1},e)$ and $\|x-x^{-1}\|_2\leq \rho_1(x,x^{-1})$.\\

\noindent{\bf Extending the metric.}\\
Suppose we have defined $\Norm_n$ for some even $n\geq 2$. We now define $\rho_{n+1}$ on $X_{n+1}$. First we inductively define an auxiliary function $\delta$ on $X^2_{n+1}$. Fix a pair $x,y\in X_{n+1}$. If $r(x)=r(y)=0$, then we set $$\delta(x,y)=\min\{\|a_1-b_1\|_n+\ldots+\|a_m-b_m\|_n:$$ $$x=(a_1\ldots a_m)',y=(b_1\ldots b_m)',\forall i\leq m (a_i,b_i,a_i-b_i\in X_n)\}.$$ Note that the minimum is indeed attained as $X_n$ is finite. Note again that for any $z\in X_{n+1}$ if $r(z)>0$ we have that either $z$ or $z^{-1}$ belongs to $X_n$. So if $r(x)>0$, $r(y)>0$, then we set $$\delta(x,y)=\begin{cases}
\|x-y\|_n & \text{if }x,y\in X_n\\
\|x^{-1}-y^{-1}\|_n & \text{if }x^{-1},y^{-1}\in X_n\\
\min\{\|x-z\|+\|z^{-1}-y^{-1}\|:z,z^{-1}\in X_n\} & \text{if }x, y^{-1}\in X_n\\
\min\{\|x^{-1}-z^{-1}\|+\|z-y\|:z,z^{-1}\in X_n\} & \text{if }x^{-1},y\in X_n.\\
\end{cases}$$

Now we suppose that for one of the elements, say $x$, we have $r(x)=0$, and for the other one we have $r(y)>0$. The following is done by induction on $r(y)$. First suppose that $y=\alpha_1 z_1+\ldots+\alpha_m z_m$, where $m>1$, $\alpha_i>0$, for all $i$, and $\alpha_1+\ldots+\alpha_m=1$. Then we set $$\delta(x,y)=\alpha_1\cdot\delta(x,z_1)+\ldots+\alpha_m\cdot\delta(x,z_m).$$ Similarly, if $y=(\alpha_1 z_1+\ldots+\alpha_m z_m)^{-1}$, where $m>1$, $\alpha_i>0$, for all $i$, and $\alpha_1+\ldots+\alpha_m=1$, then we set $$\delta(x,y)=\alpha_1\cdot\delta(x,z^{-1}_1)+\ldots+\alpha_m\cdot\delta(x,z^{-1}_m).$$

We are now ready to define $\rho_{n+1}$. Thus fix now again a pair $x,y\in X_{n+1}$ and set $$\rho_{n+1}(x,y)=\min\{\delta(a_1,b_1)+\ldots+\delta(a_m,b_m):$$ $$x=(a_1\ldots a_m)',y=(b_1\ldots b_m)',a_1,\ldots,a_m,b_1,\ldots,b_m\in X_{n+1}\}.$$

First notice that since $X_{n+1}$ is finite the minimum is attained, thus in particular we have for $x\neq y$ that $\rho_{n+1}(x,y)>0$. Since $\rho_{n+1}$ is clearly symmetric we get it satisfies the condition \eqref{con1}.

We claim that $\rho_{n+1}$ is the greatest function satisfying:\\
\begin{enumerate}[(a)]
\item\label{mcon_a} $\rho_{n+1}(x,y)\leq \|x-y\|_n$ for every $x,y\in X_n$ such that $x-y\in X_n$,
\item\label{mcon_b} $\rho_{n+1}(x,y)=\rho_{n+1}(x^{-1},y^{-1})$ for every $x,y\in X_{n+1}$,
\item\label{mcon_c} $\rho_{n+1}(ab,cd)\leq \rho_{n+1}(a,c)+\rho_{n+1}(b,d)$ for every $a,b,c,d\in X_{n+1}$ such that $ab,cd\in X_{n+1}$,
\item\label{mcon_d} $\rho_{n+1}(x,(\alpha_1 z_1+\ldots+\alpha_m z_m))\leq \alpha_1\cdot\rho_{n+1}(x,z_1)+\ldots+\alpha_m\cdot \rho_{n+1}(x,z_m)$, where $m>1$, $\alpha_i>0$, for all $i$, $\alpha_1+\ldots+\alpha_m=1$ and $\alpha_1 z_1+\ldots+\alpha_m z_m\in X_{n+1}$,
\item\label{mcon_e} $\rho_{n+1}(x,(\alpha_1 z_1+\ldots+\alpha_m z_m)^{-1})\leq \alpha_1\cdot\rho_{n+1}(x,z^{-1}_1)+\ldots+\alpha_m\cdot \rho_{n+1}(x,z^{-1}_m)$, where $m>1$, $\alpha_i>0$, for all $i$, $\alpha_1+\ldots+\alpha_m=1$ and $(\alpha_1 z_1+\ldots+\alpha_m z_m)^{-1}\in X_{n+1}$.\\

\end{enumerate}
First of all, it is clear from the definitions of $\rho_{n+1}$ (and of $\delta$) that $\rho_{n+1}$ satisfies all these conditions. Thus in particular, we get that $\rho_{n+1}$ satisfies the conditions \eqref{con3},\eqref{con4} and \eqref{con6}. Next, if $\xi$ is any other function satisfying conditions \eqref{mcon_a}-\eqref{mcon_e}, then it is readily checked that $\xi\leq \delta$ and because of \eqref{mcon_c} also $\xi\leq \rho_{n+1}$.

We shall now conclude from that that $\rho_{n+1}$ also satisfies \eqref{con2}. Indeed, let $X'_n\subseteq X_n$ be such that for every $x,y\in X'_n$ we have $x-y\in X_n$. Then consider the metric $\xi$ on $X'_n$ defined as $\xi(x,y)=\|x-y\|_n$. We claim it satisfies the conditions \eqref{mcon_a}-\eqref{mcon_e} above. Condition \eqref{mcon_a} is satisfied since $\xi(x,y)=\|x-y\|_n$ for appropriate $x,y$. Take some $x,y\in X'_n$ such that $x^{-1},y^{-1}\in X'_n$. Necessarily $x,y,x^{-1},y^{-1}\in X_{n-1}$ and since $\Norm_n$ extends $\rho_{n-1}$ we get $\xi(x,y)=\rho_{n-1}(x,y)=\rho_{n-1}(x^{-1},y^{-1})=\xi(x^{-1},y^{-1})$. Take now some $a,b,c,d\in X'_n$ such that $ab,cd\in X'_n$. We again necessarily have that $a,b,c,d,ab,cd\in X_{n-1}$ and since $\Norm_n$ extends $\rho_{n-1}$ we again obtain $\xi(ab,cd)=\rho_{n-1}(ab,cd)\leq \rho_{n-1}(a,c)+\rho_{n-1}(b,d)=\xi(a,c)+\xi(b,d)$. We have verified conditions \eqref{mcon_b} and \eqref{mcon_c}. Condition \eqref{mcon_d} follows since $\Norm_n$ satisfies the condition \eqref{con5} further above. Finally, take some $x,(\alpha_1 z_1+\ldots+\alpha_m z_m)^{-1}\in X'_n$, where $m>1$, $\alpha_i>0$, for all $i$, $\alpha_1+\ldots+\alpha_m=1$. We have $\xi(x,(\alpha_1 z_1+\ldots+\alpha_m z_m)^{-1})=\|x-(\alpha_1 z_1+\ldots+\alpha_m z_m)^{-1}\|_n\leq \alpha_1\cdot\|x-z_1^{-1}\|_n+\ldots+\alpha_m\cdot\|x-z_m\|_n=\alpha_1\cdot\xi(x,z_1)+\ldots+\alpha_m\cdot\xi(x,z_m)$, where the middle inequality follows from the property \eqref{con6} above.

We thus get $\|x-y\|_n=\xi(x,y)\leq \rho_{n+1}(x,y)$ for every $x,y\in X_n$ such that $x-y\in X_n$. Since by assumption $\rho_{n+1}(x,y)\leq \|x-y\|_n$ we get $\rho_{n+1}(x,y)=\|x-y\|_n$ and we are done.\\

\noindent{\bf Extending the norm.}\\
Now suppose we have defined $\rho_n$ on $X_n$ for some odd $n>2$. We define $\Norm_{n+1}$ on $X_{n+1}$. We again at first define, inductively, an auxiliary function $\gamma :X_{n+1}\rightarrow \Rea$. First, for every $x,y\in X_n$ we set $$\gamma(x-y)=\rho_n(x,y).$$ Next, for every $x,y\in X_{n+1}$ such that $x-y\in X_{n+1}$ and $x\in X_{n+1}\setminus X_n$ we define $\gamma(x-y)$ by induction on $r(y)$. If $r(y)=0$ then we set $$\gamma(x-y)=\min\{|\beta_1|\cdot\gamma(v_1)+\ldots+|\beta_i|\cdot\gamma(v_i):$$ $$x-y=\beta_1 v_1+\ldots+\beta_i v_i,\forall j\leq i \exists a_j,b_j\in X_n(v_j=a_j-b_j)\}.$$

If $r(y)>0$ and $y=(\alpha_1 z_1+\ldots+\alpha_m z_m)^{-1}$, where $m>1$, $\alpha_i>0$, for every $i$, and $\alpha_1+\ldots+\alpha_m=1$, then we set $$\gamma(x-y)=\alpha_1\cdot\gamma(x-z^{-1}_1)+\ldots+\alpha_m\cdot\gamma(x-z^{-1}_m).$$

Finally, for any $x\in X_{n+1}$ we set $$\|x\|_{n+1}=\min\{\gamma(y_1)+\ldots+\gamma(y_i):x=y_1+\ldots+y_i,y_1,\ldots,y_i\in X_{n+1}\}.$$

First thing to observe is again that for any $x\in X_{n+1}$ we have $\|x\|_{n+1}=0$ iff $x=e$. It follows that condition $(1)$ is satisfied for $\Norm_{n+1}$.

Next we claim that $\Norm_{n+1}$ is the greatest function satisfying:
\begin{enumerate}[(a)]
\item\label{ncon_a} $\|x-y\|_{n+1}\leq \rho_n(x,y)$ for every $x,y\in X_n$,
\item\label{ncon_b} $\|\alpha x+\beta y\|_{n+1}\leq \|\alpha x\|_{n+1}+\|\beta y\|_{n+1}= |\alpha|\cdot \|x\|_{n+1}+|\beta|\cdot\|y\|_{n+1}$ for $x,y\in X_{n+1}$ such that also $\alpha x+\beta y\in X_{n+1}$,
\item\label{ncon_c} $\|x-(\alpha_1 z_1+\ldots+\alpha_m z_m)^{-1}\|_{n+1}\leq \alpha_1\cdot\|x-z^{-1}_1\|_{n+1}+\ldots+\alpha_m\cdot\|x-z^{-1}_m\|_{n+1}$, where $m>1$, $\alpha_i>0$, for every $i$, and $\alpha_1+\ldots+\alpha_m=1$.

\end{enumerate}

First, it follows from the definitions of $\gamma$ and $\Norm_{n+1}$ that $\Norm_{n+1}$ satisfies all these conditions, thus we have that it satisfies the conditions \eqref{con5} and \eqref{con6} further above. If $\xi$ is any other functions satisfying conditions \eqref{ncon_a}-\eqref{ncon_c} then it again follows that necessarily $\xi\leq \gamma$ and thus also $\xi\leq\Norm_{n+1}$.

We are ready to verify the remaining condition \eqref{con2} for $\Norm_{n+1}$ that it extends $\rho_n$. Let $X'_{n+1}\subseteq X_{n+1}$ be the set $\{x-y:x,y\in X_n\}$. Define $\xi$ on $X'_{n+1}$ as follows: $\xi(x-y)=\rho_n(x,y)$ for every $x-y\in X'_{n+1}$. Then $\xi$ satisfies \eqref{ncon_a} since $\xi(x-y)=\rho_n(x,y)$ for appropriate $x,y$. Next we check \eqref{ncon_b}. Take some $\alpha x+\beta y$. If $\alpha=1$ and $\beta=-1$ (or vice versa), then we have $\xi(x-y)=\rho_n(x,y)\leq \rho_n(x,e)+\rho_n(e,y)=\xi(x)+\xi(y)$, where the middle inequality follows from the condition \eqref{con3} that $\rho_n$ satisfies (which implies the triangle inequality). Finally, for $(\alpha_1 z_1+\ldots+\alpha_m z_m)^{-1}$, where $m>1$, $\alpha_i>0$, for every $i$, and $\alpha_1+\ldots+\alpha_m=1$, we have $\xi(x-(\alpha_1 z_1+\ldots+\alpha_m z_m)^{-1})=\rho_n(x,(\alpha_1 z_1+\ldots+\alpha_m z_m)^{-1})\leq \alpha_1\cdot\rho_n(x,z^{-1}_1)+\ldots+\alpha_m\cdot\rho_n(x,z^{-1}_m)= \alpha_1\cdot\xi(x-z^{-1}_1)+\ldots+\alpha_m\cdot\xi(x-z^{-1}_m)$, since $\rho_n$ satisfies the condition \eqref{con6}, and we are done.\\

Now we can define $\Norm$ on $X$ by putting $$\Norm=\bigcup_i \Norm_i$$ and analogously we can define $\rho$ on $X$ by putting $$\rho=\bigcup_i \rho_i.$$

By the condition \eqref{con2} we have that for any $x,y\in X$ we have $\|x-y\|=\rho(x,y)$. By \eqref{con1} we have that for $x\neq y\in X$ we have $\rho(x,y)>0$ and equivalently, for any $x\neq e\in X$ we have $\|x\|>0$.

Let us check that $\rho$ is a bi-invariant metric on $X$ when considered as a (free) group. By \eqref{con1} is symmetric. We use the following simple fact.
\begin{fact}
Let $G$ be a group equipped with a symmetric function $d:G^2\rightarrow \Rea^+_0$ that is equal to $0$ only on diagonal. Then $d$ is a bi-invariant metric if and only if for every $x,y,v,w\in G$ we have $d(x\cdot y,v\cdot w)\leq d(x,v)+d(y,w)$.
\end{fact}
\begin{proof}
If $d$ is a bi-invariant metric then the inequality readily follows from bi-invariance and using triangle inequality.

So suppose that $d$ satisfies such an inequality for every $x,y,v,w\in G$. Fix $a,b,c\in G$. Then $d(a,c)=d(a\cdot b^{-1}\cdot b,b\cdot b^{-1}\cdot c)\leq d(a,b)+d(b^{-1},b^{-1})+d(b,c)=d(a,b)+d(b,c)$, so $d$ is a metric. Now $d(a\cdot b,a\cdot c)\leq d(a,a)+d(b,c)=d(b,c)=d(a^{-1}\cdot a\cdot b,a^{-1}\cdot a\cdot c)\leq d(a^{-1},a^{-1})+d(a\cdot b,a\cdot c)=d(a\cdot b,a\cdot c)$ which shows the left-invariance. The right invariance is done analogously.
\end{proof}
However, $\rho$ does satisfy the condition from the statement of the fact since it satisfies the condition \eqref{con3}.\\

Similarly, for every $x,y$ and $\alpha,\beta\in D$ (recall that $D$ denotes the dyadic rationals) we have $\|\alpha x+\beta y\|\leq \|\alpha x\|+\|\beta y\|= |\alpha|\cdot\|x\|+|\beta|\cdot\|y\|$ since $\Norm$ satisfies the condition $(5)$.\\

Denote now by $\mathbb{X}$ the completion of $X$ with respect to $\rho$, or equivalently, with respect to $\Norm$. Both the multiplicative and additive group operations extend to the completion as well as the scalar multiplication by dyadic rationals. Moreover, since the dyadic rationals are dense in $\Rea$, $\mathbb{X}$ has well-defined scalar multiplication by all the reals. Thus $\mathbb{X}$ is a Banach space.

\section{Concluding discussion and questions}
Let us comment on similarities between the constructions of the presented Banach group and Lipschitz-free Banach spaces. For any pointed metric space $(X,0)$ there is a Banach space $LF(X)$, called Lipschitz-free space (or Arens-Eells space) over $(X,0)$, that has elements of $X\setminus\{0\}$ as the Hamel basis and the point $0$ as the Banach space zero. $LF(X)$ contains $X$ isometrically and the norm of $LF(X)$ is uniquely described as the largest norm on the linear span of $X$ that extends the metric on $X$.

By similar means, one can define the largest bi-invariant metric on $F_{X\setminus\{0\}}$, the free group having $0$ as a neutral element and $X\setminus\{0\}$ as the set of free generators, that extends the metric on $X$.

The following generalization of the Lipschitz-free construction was given in \cite{BY}. If $Y$ is a Banach space and $Y\subseteq X$ is a metric extension of $Y$ such that for every $x\in X$ the distance function $\mathrm{dist}(x,\cdot):Y\rightarrow \Rea$ is convex, then there is a free space $LF_Y(X)$ which has $Y$ as a subspace, contains $X$ as a subset so that the elements of $X\setminus Y$ are linearly independent, and the norm on $LF_Y(X)$ is the largest norm on such a vector space that extends the metric on $X$.

An analogous generalization of the Graev metric on groups was given in \cite{Do}; i.e. for any group $G$ with bi-invariant metric and any metric extension $G\subseteq X$ such that $G$ is closed in $X$ there exists the largest bi-invariant metric on the free product $G\ast F_{X\setminus G}$.

Roughly speaking, the idea behind our construction from the present paper was to alternatively use those two constructions. That means, to start with some pointed metric space $(X,0)$. Then consider the Lipschitz-free space $LF(X)$, then the free group $F_{LF(X)}$ with the Graev metric, then $LF_{LF(X)}(F_{LF(X)})$, etc., and at the end to take the completion. The reason why this approach does not literally work is that for a general point $x$ in $F_{LF(X)}$, the distance function $\mathrm{dist}(x,\cdot):LF(X)\rightarrow\Rea$ is not convex. The remedy was to use `finitary' versions of the constructions above where we were able to guarantee the convexity of the appropriate distance functions.
\subsection{Freeness of the Banach group}
We shall argue that although $\mathbb{X}$ as a Banach space is likely not one of the ``well-behaved" spaces it is not completely random either. It can be uniquely characterized as a free one-generated uniform Banach group whose metric induced by the norm is bi-invariant. That can be described via certain universal property. Let us at first define morphisms in the category of uniform Banach groups. The natural definition seems to be a bounded linear operator that is moreover a group homomorphism in the additional group structure. Note that it is then automatically uniformly continuous group homomorphism.

We claim that:
\begin{thm}
$\mathbb{X}$ is the free uniform Banach group over one generator whose metric induced by the norm is bi-invariant. That means, for any uniform Banach group $\mathbb{Y}$ whose metric induced by the norm is bi-invariant, and any element $y\in \mathbb{Y}$ there exists a unique uniform Banach group morphism $\phi:\mathbb{X}\rightarrow \mathbb{Y}$ such that $\phi(x)=y$ and $\|\phi\|=\|y\|_{\mathbb{Y}}$.

This property characterizes $\mathbb{X}$ uniquely up to isomorphism.
\end{thm}
\begin{remark}
Note that every uniform Banach group $Z$ admits a bi-invariant metric that is uniformly continuous with respect to norm. For any $y,z\in Z$ set $D(y,z)=\sup_{v,w\in Z} \|v\cdot y\cdot w-v\cdot z\cdot w\|_Z$. It is always satisfied that $D(y,z)\geq \|y-z\|$ and the inequality may be strict in general.
\end{remark}
\begin{proof}
If $y=0$ then $\phi$ is the zero morphism and there is nothing to prove. So we assume that $y\neq 0$.

Let us take a closer look on the countable dense set $X\subseteq \mathbb{X}$. One can see that each element of $X$ is obtained in a unique way from the element $x$ (the only element of $S_1$) using operations of addition $+$, multiplication $\cdot$, additive and multiplicative inverses $-,{}^{-1}$ and by scalar multiplication by dyadic rationals $D$. It follows that there is a unique map $\phi':X\rightarrow \mathbb{Y}$ that preserves those operations and satisfies $\phi'(x)=y$. We need to show that for every $z\in X$ we have
\begin{equation}\label{max_norm}
\|\phi'(z)\|_{\mathbb{Y}}\leq \|y\|_{\mathbb{Y}}\|z\|_{\mathbb{X}}.
\end{equation}
Once this is proved we can (uniquely) extend $\phi'$ to the completion of $X$ to obtain $\phi\supseteq \phi':\mathbb{X}\rightarrow \mathbb{Y}$. $\phi$ is still linear and preserves the additional group structure. By \ref{max_norm}, we get that $\|\phi\|\leq \|y\|_{\mathbb{Y}}$. On the other hand, $\phi(x)=y$, thus $\|\phi\|=\|y\|_{\mathbb{Y}}$.\\

To simplify the notation, we shall without loss of generality assume that $\|y\|_{\mathbb{Y}}=1$.

We consider the following pseudometric $\sigma$, resp. pseudonorm $\||\cdot\||$, on $X$. For any $y,z\in X$ we set $\sigma(y,z)=\|\phi'(y)-\phi'(z)\|_{\mathbb{Y}}$; similarly, for every $y\in X$ we set $\||y\||=\|\phi'(y)\|_{\mathbb{Y}}$. For every odd $n$ we denote by $\sigma_n$ the restriction of $\sigma$ to $X_n$. Similarly, for every even $n$ we denote by $\||\cdot\||_n$ the restriction of $\||\cdot\||$ to $X_n$. It suffices to show, by induction, that for every odd $n$ we have $\sigma_n\leq \rho_n$ and for every even $n$ we have $\||\cdot\||_n\leq \Norm _n$.

Consider the case $n=1$. We have (from bi-invariance) $\sigma_1(x,e)=\sigma_1(x^{-1},e)=\rho_1(x,e)=\rho_1(x^{-1},e)=1$. It follows from the triangle inequality that $\sigma_1(x,x^{-1})\leq 2=\rho_1(x,x^{-1})$.

Consider now the case $n=2$. Take any $\alpha,\beta\in D_1$. We have $\|\alpha x+\beta x^{-1}\|_2=\min\{|\gamma_1|\cdot \rho_1(x,e)+|\gamma_2|\cdot \rho_1(x^{-1},e)+|\gamma_3|\cdot \rho_1(x,x^{-1}):\alpha x+\beta x^{-1}=\gamma_1 x+\gamma_2 x^{-1}+\gamma_3(x-x^{-1}),\gamma_1,\gamma_2,\gamma_3\in\Rea\}$. However, for any $\gamma_1,\gamma_2,\gamma_3\in\Rea$ if $\alpha x+\beta x^{-1}=\gamma_1 x+\gamma_2 x^{-1}+\gamma_3(x-x^{-1})$, then by the subadditivity and homogeneity of the norm we must have $$\||\alpha x+\beta x^{-1}\||_2=\leq |\gamma_1|\cdot \sigma_1(x,e)+|\gamma_2|\cdot \sigma_1(x^{-1},e)+|\gamma_3|\cdot \sigma_1(x,x^{-1})\leq$$ $$|\gamma_1|\cdot \rho_1(x,e)+|\gamma_2|\cdot \rho_1(x^{-1},e)+|\gamma_3|\cdot \rho_1(x,x^{-1}).$$
Thus $\||\cdot\||_2\leq \Norm_2$.

Consider now some general odd $n$. Necessarily, $\sigma_{n+1}$ satisfies all the following inequalities (recall the analogous inequalities for $\rho_{n+1}$)
\begin{enumerate}[(a)]
\item\label{smcon_a} $\sigma_{n+1}(x,y)\leq \||x-y\||_n$ for every $x,y\in X_n$ such that $x-y\in X_n$,
\item\label{smcon_b} $\sigma_{n+1}(x,y)=\sigma_{n+1}(x^{-1},y^{-1})$ for every $x,y\in X_{n+1}$,
\item\label{smcon_c} $\sigma_{n+1}(ab,cd)\leq \sigma_{n+1}(a,c)+\sigma_{n+1}(b,d)$ for every $a,b,c,d\in X_{n+1}$ such that $ab,cd\in X_{n+1}$,
\item\label{smcon_d} $\sigma_{n+1}(x,(\alpha_1 z_1+\ldots+\alpha_m z_m))\leq \alpha_1\cdot\sigma_{n+1}(x,z_1)+\ldots+\alpha_m\cdot \sigma_{n+1}(x,z_m)$, where $m>1$, $\alpha_i>0$, for all $i$, $\alpha_1+\ldots+\alpha_m=1$ and $\alpha_1 z_1+\ldots+\alpha_m z_m\in X_{n+1}$,
\item\label{smcon_e} $\sigma_{n+1}(x,(\alpha_1 z_1+\ldots+\alpha_m z_m)^{-1})\leq \alpha_1\cdot\sigma_{n+1}(x,z^{-1}_1)+\ldots+\alpha_m\cdot \rho_{n+1}(x,z^{-1}_m)$, where $m>1$, $\alpha_i>0$, for all $i$, $\alpha_1+\ldots+\alpha_m=1$ and $(\alpha_1 z_1+\ldots+\alpha_m z_m)^{-1}\in X_{n+1}$.\\

\end{enumerate}
Inequalities \eqref{smcon_a},\eqref{smcon_b} and \eqref{smcon_c} are clear; \eqref{smcon_d} follows since $\sigma_{n+1}(x,(\alpha_1 z_1+\ldots+\alpha_m z_m))=\|\phi'(x-(\alpha_1 z_1+\ldots+\alpha_m z_m))\|_{\mathbb{Y}}\leq \alpha_1\|\phi'(x-z_1)\|_{\mathbb{Y}}+\ldots+\alpha_m\|\phi'(x-z_m)\|_{\mathbb{Y}}=\alpha_1\cdot\sigma_{n+1}(x,z_1)+\ldots+\alpha_m\cdot \sigma_{n+1}(x,z_m)$; \eqref{smcon_e} follows analogously.

Since $\rho_{n+1}$ has been shown to be the greatest function satisfying the analogous conditions \eqref{mcon_a}-\eqref{mcon_e}, it follows from the induction hypothesis that $\sigma_{n+1}\leq \rho_{n+1}$.

Finally, for a general even $n$, a completely analogous argumentation, using the fact that $\Norm_{n+1}$ was the greatest function satisfying certain conditions, gives that $\||\cdot\||_{n+1}\leq \Norm_{n+1}$.\\

The uniqueness is a standard argument using the universality property.
\end{proof}

As mentioned above, one can rightfully suspect that $\mathbb{X}$ is not one of the well-behaved Banach spaces. Actually, in our opinion, it would not be difficult to enhance the construction above so that $\mathbb{X}$ was isometric to the Gurarij space (\cite{Gu}). It thus seems natural to ask whether one can get a non-commutative Banach group modeled on a `reasonable' space.

We mentioned in the preliminary section that the Heisenberg group $UT_3^3(\mathbb{R})$, modeled on a three-dimensional Banach space, is not a uniform Banach group, although it is very closed to it (this is also the content of Remark 5.2. in \cite{PraWes}). The following question thus arises.
\begin{question}
Does there exist a finite-dimensional non-commutative uniform Banach group?
\end{question}
We expect the answer to the previous question to be negative. However, one can then ask:
\begin{question}
Does there exist a non-commutative uniform Banach group modeled on a Hilbert space?
\end{question}
As already mentioned in the introduction, uniform Banach groups were originally introduced with connection to the infinite-dimensional Hilbert's fifth problem. The following question thus also seems to be natural.
\begin{question}
Are the (non-commutative) group operations on $\mathbb{X}$ (Fr\' echet) differentiable? Does there exist a non-commutative uniform Banach group that is a Banach-Lie group?
\end{question}

\end{document}